\documentclass[11pt]{article}
\usepackage{geometry}
\usepackage{cite}
\usepackage{graphicx}
\usepackage{amsfonts}
\usepackage{amsmath} 
\usepackage{amssymb}

\usepackage{amsthm}	
\usepackage{dsfont}
\usepackage{mathrsfs}
\usepackage{enumerate}
\usepackage{xcolor}
\usepackage{booktabs}
\usepackage{subcaption}
\usepackage{bm}
\usepackage{authblk} 

\newtheorem{thm}{Theorem}

\newtheorem{cor}{Corollary}
\newtheorem{defn}{Definition}

\newtheorem{rem}{Remark}

\newcommand{\Rset}{\mathbb{R}}

\newcommand{\Mcal}{{\cal M}}

\newcommand{\Pcal}{{\cal P}}

\newcommand{\Wcal}{{\mathcal{W}}}

\newcommand{\Xcal}{{\cal X}}

\newcounter{l1}
\newcounter{l2}
\newcounter{l3}
\newcommand{\bdotlist}{\begin{list}{$\bullet$}{}}
\newcommand{\bboxlist}{\begin{list}{$\Box$}{}}
\newcommand{\bbboxlist}{\begin{list}{\raisebox{.005in}{{\tiny
$\blacksquare$ \ \ }}}{}}
\newcommand{\bdashlist}{\begin{list}{$-$}{} }
\newcommand{\blist}{\begin{list}{}{} }
\newcommand{\barablist}{\begin{list}{\arabic{l1}}{\usecounter{l1}}}
\newcommand{\balphlist}{\begin{list}{(\alph{l2})}{\usecounter{l2}}}
\newcommand{\bAlphlist}{\begin{list}{\Alph{l2}.}{\usecounter{l2}}}
\newcommand{\bdiamlist}{\begin{list}{$\diamond$}{}}
\newcommand{\bromalist}{\begin{list}{(\roman{l3})}{\usecounter{l3}}}

\newcommand{\rel}{w}
\newcommand{\ran}{\bm{w}}
\newcommand{\uset}{\mathcal{W}}
\newcommand{\supp}{\sigma}
\DeclareMathOperator*{\argmin}{arg\,min}

\title{\bf Distributionally Robust Regret Minimization}
\author{{\Large Eilyan Bitar}\thanks{School of Electrical and Computer Engineering, Cornell University, Ithaca, NY, USA. Email: eyb5@cornell.edu }}
\date{}

\begin{document}
\maketitle

\begin{abstract}
We consider decision-making problems involving the optimization of linear objective functions with uncertain coefficients. 
The probability distribution of the coefficients---which are assumed to be stochastic in nature---is unknown to the decision maker but is assumed to lie within a given ambiguity set, defined as a type-1 Wasserstein ball centered at a given nominal distribution. To account for this uncertainty, we minimize the worst-case expected \emph{regret} over all distributions in the  ambiguity set. 
Here, the (ex post) regret experienced by the decision maker  is defined as the  difference between the cost incurred by a chosen decision given a particular realization of the objective coefficients and the minimum achievable cost  with perfect knowledge of the coefficients at the outset. 
For this class of ambiguity sets,  the worst-case expected regret is shown to
equal the expected regret under the nominal distribution plus a regularization term that has the effect of drawing optimal solutions toward the ``center'' of the feasible region as the radius of the ambiguity set increases. This novel form of regularization is also shown to arise when minimizing the worst-case conditional value-at-risk (CVaR) of regret. We show that, under certain conditions,  distributionally robust regret minimization problems over type-1 Wasserstein balls can be recast as tractable finite-dimensional convex programs.  
\end{abstract}

\section{Introduction}
We consider a class of uncertain linear optimization problems in which the coefficients of the objective function are unknown, and a decision must be made without prior knowledge of these coefficients.  To handle such uncertainty in the problem data, a common approach is to assume that the unknown parameters belong to a given set, known as the uncertainty set, and 
to construct a solution that is optimal for the least-favorable parameter values in the given uncertainty set.
This corresponds to the following class of robust linear optimization problems:
\begin{align} \label{eq:RO}
    \inf_{x \in \Xcal} \, \sup_{\rel \in \uset} \  \rel^\top x.
\end{align}
In problem \eqref{eq:RO}, the decision vector  $x \in \Rset^n$ is constrained to lie within a  feasible set  $\Xcal \subseteq \Rset^n$ and and the vector of uncertain parameters $w \in \Rset^n$ is confined to an uncertainty set $\Wcal \subseteq \Rset^n$.

Robust optimization problems, as defined in \eqref{eq:RO}, are known to be computationally tractable under certain conditions. For example, if the feasible set $\Xcal$ is polyhedral and the uncertainty set $\Wcal$ is polyhedral or ellipsoidal, then the problem \eqref{eq:RO} can be reformulated as a polynomially sized linear program or a second-order cone program, respectively \cite{ben1999robust, ben2009robust}. However, despite its broad appeal, a potential drawback of the robust optimization framework is its inherent conservatism. By optimizing for the worst-case scenario within the uncertainty set, robust optimization methods can lead to overly cautious decisions that may underperform when facing less severe realizations of the uncertain parameters.

To mitigate this conservatism, robust regret minimization has been explored as an alternative to the robust optimization approach. Originally proposed in \cite{savage1951theory},  regret measures the difference between the cost realized by a chosen decision  and the best possible cost that could have been achieved had the uncertain parameters been known at the outset. Mathematically, the regret incurred by a decision $x \in \Xcal$ given a particular realization of the uncertain parameters $ w\in \Wcal$ is defined as
\begin{align}
    R(x, \,\rel) := \rel^\top x - \inf_{y \in \Xcal} \rel^\top y. 
\end{align}
With this definition, the robust regret minimization problem can be expressed as
\begin{align} \label{eq:MRO}
    \inf_{x \in \Xcal} \, \sup_{\rel \in \uset}  \ R(x,\, \rel).
\end{align}
We note that a related class of robust regret minimization problems, where uncertainty is confined to the constraints on $x$, is known to be computationally tractable. Specifically, if  the feasible set $\Xcal$  and the uncertainty set $\Wcal$ are both polyhedral, and the uncertain parameters $w$ only appear linearly in the right-hand side of the constraints on $x$, then  the corresponding robust regret minimization problem can be reformulated as an equivalent linear program \cite{gabrel2010robustness, bertsimas2020relative}. In contrast, robust regret minimization problems with uncertain coefficients in the objective function, as defined in \eqref{eq:MRO}, are more difficult to solve. In particular, problem \eqref{eq:MRO} has been proven to be strongly NP-hard when the uncertainty set is an axis-aligned hyperrectangle \cite{averbakh2005complexity}. As a result, most of the algorithms that have been developed to solve problem \eqref{eq:MRO} are either conservative in nature or computationally expensive.\footnote{We refer the reader to \cite{poursoltani2022adjustable} for a comprehensive overview of the various algorithms that have been developed to tackle this class of single-stage robust regret minimization problems.}  
There are, however, special cases of  the minimax regret optimization problem \eqref{eq:MRO} that can be solved in polynomial time, including specific resource allocation problems with rectangular uncertainty sets \cite{averbakh2004minmax, conde2005complexity}, and certain inventory management problems with simplex uncertainty sets \cite{poursoltani2022adjustable}.

\subsection{A Distributionally Robust Approach to Regret Minimization}
In this paper, we approach the problem of regret minimization from the perspective of distributionally robust optimization. Unlike the robust formulation in \eqref{eq:MRO}, where  the unknown parameters are allowed to take any value within a predefined uncertainty set,  the distributionally robust approach treats the uncertain parameters as a random vector $\ran$ whose distribution $P$ is unknown, but  assumed to lie within a given ambiguity set $\Pcal$---a predefined family of possible distributions. 
Using the given ambiguity set, decisions are constructed to minimize the worst-case expected regret  across all possible distributions within the ambiguity set. Mathematically, this leads to the following distributionally robust regret minimization  problem considered in this paper:
    \begin{align} \label{eq:DRRO}
        \inf_{x \in \Xcal} \, \sup_{P \in \Pcal}  \, \mathds{E}_P\left[{ R(x, \, \ran)} \right].
    \end{align}

It is important to note that, in problem \eqref{eq:DRRO}, regret is evaluated \emph{ex post}, meaning it is assessed after the uncertain parameters have been realized. This differs from the \emph{ex ante} approach to distributionally robust regret minimization,  where the ex ante regret of a decision is defined as the difference between the \emph{expected cost} incurred by that decision and the minimum expected cost, assuming perfect knowledge of the underlying probability distribution.  
Distributionally robust \emph{ex ante} regret minimization has been extensively studied  
in the context of the newsvendor model \cite{chen2021regret,perakis2008regret,yue2006expected,zhu2013newsvendor},  two-stage stochastic programs \cite{cho2024wasserstein}, and  multi-stage stochastic programs defined over discrete probability spaces \cite{poursoltani2024risk}.

The formulation in \cite{natarajan2014probabilistic} is more closely aligned with the approach taken in this paper, as it  employs an \emph{ex post} measure of regret. 
Specifically, they define the ambiguity set as the collection of all distributions with  common support  (a hyperrectangle),  mean, and higher-order marginal moments. With this ambiguity set, they consider a risk-sensitive variant of the optimization model in \eqref{eq:DRRO}, where decisions are chosen to minimize the  worst-case conditional value-at-risk (CVaR) of the ex post regret.  However, since the moments of a distribution are typically estimated from data and, therefore, may themselves be uncertain, 
we consider an alternative class of (distance-based) ambiguity sets that do not impose such rigid restrictions a priori. Specifically, in this paper, the ambiguity set is taken to be the type-1 Wasserstein ball distributions centered at a given nominal distribution (e.g., an empirical distribution constructed from measured data), where the radius of this ball represents the  decision maker's confidence  in the accuracy of the nominal distribution. 

While the focus of this  paper is on single-stage stochastic programs,  we note that
Wasserstein distance-based ambiguity sets have also been utilized to synthesize distributionally robust (ex post) regret optimal control polices for unconstrained stochastic linear-quadratic optimal control problems, where ambiguity in the process disturbance and measurement noise distributions is modeled using  type-2 Wasserstein balls \cite{al2023distributionally, hajar2023wasserstein, kargin2024wasserstein}.

\subsection{Main Results and Paper Outline}
In Section \ref{sec:DRRM}, we introduce the specific class of distributionally robust (ex post) regret minimization problems \eqref{eq:DRRO} considered in this paper, where the ambiguity set $\Pcal$ is defined as a type-1 Wasserstein ball, and the feasible region $\Xcal$ is assumed to be  nonempty, compact, and  possibly nonconvex. In Section \ref{sec:reform}, we show that the worst-case expected regret of a decision $x \in \Xcal$ can be expressed as the sum of the expected regret under the nominal distribution and a regularization term. This regularization term, which scales with the radius of the ambiguity set,  penalizes the distance of the decision $x$ to the farthest point in the feasible set.  This  has the effect of drawing optimal solutions to problem \eqref{eq:DRRO} toward the ``center'' of the feasible region as the Wasserstein radius increases. 

Based on this reformulation, the Wasserstein distributionally robust regret minimization problem \eqref{eq:DRRO} is shown to ``interpolate'' between the nominal expected cost minimization problem, as the radius of the ambiguity set approaches zero, and the robust regret minimization problem \eqref{eq:MRO} with a norm-bounded uncertainty set, as the radius of the ambiguity set  grows large. In Section \ref{sec:cvar}, this specific form of regularization is also shown to emerge when considering a risk-sensitive variant of problem \eqref{eq:DRRO}, where decisions are selected to  minimize the worst-case conditional value-at-risk of regret over all distributions in the type-1 Wasserstein ball. 

Finally, while the regularization function is shown to be NP-hard to compute in general (it is equivalent to a norm maximization problem over the feasible set $\Xcal$), in Section \ref{sec:tractable}, we identify conditions on the structure of the feasible set  and the choice of norm used in the definition of the Wasserstein distance that enable the equivalent reformulation of distributionally robust regret minimization problem \eqref{eq:DRRO_reduction} as a tractable convex program.  

\subsection{Notation}
Given a norm $\| \cdot\|$ on $\Rset^n$, the dual norm is defined as  $\|x\|_* := \sup\{ x^\top y \, | \, \|y\| \leq 1\}$ for all $x \in \Rset^n$.  
The support function associated with a set $\Xcal \subseteq \Rset^n$ is defined as $\sigma_{\Xcal}(y) := \sup_{x \in \Xcal} x^\top y $ for all $y \in \Rset^n$. 
 We use boldface symbols to denote random variables, and non-boldface symbols to denote particular values in the range of a random variable and other deterministic quantities.

\section{Distributionally Robust Expected Regret Minimization}  \label{sec:DRRM}

In this section, we employ a duality theory for optimal transport problems to show that the distributionally robust expected regret minimization problem \eqref{eq:DRRO}, with a type-1 Wasserstein distance-based ambiguity set,  is equivalent to a regularized version of the expected cost minimization problem under the nominal distribution.
Interestingly, the resulting form of regularization is shown to encourage solutions that are ``centrally positioned'' within the feasible set $\Xcal$. In Section \ref{sec:tractable}, we establish conditions on the structure of the feasible set and the choice of norm used in defining the Wasserstein distance that enable the reformulation of problem \eqref{eq:DRRO} as a finite-dimensional convex program. In Section \ref{sec:comparison}, we draw a comparison between this framework and the more widely studied distributionally robust expected cost minimization approach.

\subsection{Strong Duality for Distributionally Robust Regret Minimization Problems} \label{sec:reform}
Unless stated otherwise, we let $\| \cdot\|$ be an arbitrary norm on $\Rset^n$, and denote by
 $\Mcal(\Rset^{n})$  the collection of Borel probability measures $P$ on $\Rset^{n}$ that satisfy $\mathds{E}_P\left[ \|\ran \|  \right]  < \infty $.
The  type-1 Wasserstein distance between probability measures is defined as follows.

\begin{defn}[Type-1 Wasserstein Distance] \rm \label{defn:wass_dist} The \emph{type-1 Wasserstein distance} between two distributions $P_1, P_2 \in \Mcal(\Rset^n)$  is defined as 
\begin{align*}
    W_1(P_1, \, P_2) := \hspace{-.075in}\inf_{\pi \in \Pi(P_1, \, P_2)}  \int_{\Rset^{N_x} \times \Rset^{N_x}} \hspace{-.05in} \| z_1 - z_2 \| \pi(d z_1, \, d z_2),
\end{align*}
 where $\Pi(P_1, \, P_2)$ denotes the set of all joint distributions  in $\Mcal(\Rset^n \times \Rset^n)$ with marginal distributions $P_1$ and $P_2$.
\end{defn}
Given a \emph{nominal distribution} $P_0 \in \Mcal( \Rset^n)$, we define the ambiguity set $\Pcal$ as the set of all distributions 
whose type-1 Wasserstein distance to $P_0$ is at most $r \geq 0$, i.e., 
\begin{align} \label{eq:amb_set}
    \Pcal := \{ P \in \Mcal( \Rset^n) \, | \, W_1(P, P_0) \leq r \}.
\end{align}
The \emph{radius of the ambiguity set}, $r$, can be interpreted as encoding the decision maker's level of confidence in how accurately the nominal distribution reflects the true underlying distribution. In particular, the ambiguity set becomes a singleton  containing only the nominal distribution if the ambiguity set radius is decreased to zero.

Given this specific family of ambiguity sets, we now turn our attention to the inner maximization (worst-case expectation) in the distributionally robust regret minimization problem \eqref{eq:DRRO}.  Strong duality has been shown  to hold for a broad family of Wasserstein worst-case expectation problems \cite{blanchet2019quantifying, gao2022distributionally, mohajerin2018data}.
In particular, the following result establishes conditions under which strong duality is guaranteed to hold, along with a characterization of the corresponding dual optimization problem.

\begin{thm}[Strong Duality, \cite{gao2022distributionally}] \label{thm:strong_duality} \rm Let $r>0$ and let  $f: \Rset^{n} \rightarrow \Rset$ be a Borel-measurable function such that $\mathds{E}_{P_0}[|f(\ran)|] < \infty$.
It holds that
\begin{align} \label{eq:strong_duality}
   \sup_{P \in \Pcal} \mathds{E}_P[f(\ran)]  \,  = \, \inf_{\lambda \geq 0} \mathds{E}_{P_0}  \bigg[ \sup_{z \in \Rset^n}  \big\{  f(z)  + \lambda (r - \| z - \ran\|)    \big\}    \bigg].
\end{align}
\end{thm}
We utilize Theorem \ref{thm:strong_duality} to derive an equivalent dual reformulation for the worst-case expected regret over the type-1 Wasserstein ball of distributions. 

\begin{thm}[Worst-Case Expected Regret] \label{thm:DRRO} \rm Let $r \geq 0$  and assume that the feasible set $\Xcal \subset \Rset^n$ is nonempty and compact. The worst-case expected regret incurred by a decision $x \in \Xcal$ can be expressed as:
\begin{align} \label{eq:worst_case_expec_reform}
    \sup_{P \in \Pcal}  \, \mathds{E}_P\left[{ R(x, \, \ran)} \right] \,  = \, \mathds{E}_{P_0}[R(x, \, \ran) ] \, + \,  r \sup_{v \in \Xcal} \|  x - v \|_* .
\end{align}
Furthermore, a point $x \in \Xcal$ is an optimal solution to the distributionally robust regret minimization problem \eqref{eq:DRRO} if and only if it is an optimal solution to the following problem:
   \begin{align} \label{eq:DRRO_reduction}
\inf_{x \in \Xcal}  \Big\{ \mathds{E}_{P_0}[\ran^\top x]   \ + \ r \sup_{v \in \Xcal} \|x - v\|_*   \Big\}.
\end{align} 
\end{thm}

    \begin{proof}  When $r = 0$, the result is immediate, as the ambiguity set reduces to a singleton containing only the nominal distribution $P_0$.  For $r > 0$, we begin by proving the identity stated in Equation \eqref{eq:worst_case_expec_reform}.
    First, observe that for each $x \in \Xcal$,  the regret function $R(x, w) = w^\top x - \inf_{y \in \Xcal} w^\top y $ is convex in $w$ over $\Rset^n$ and, as a result, is Borel measurable.  Also, since it is assumed that  $P_0 \in \Mcal(\Rset^n)$ and that the feasible set  $\Xcal$ is bounded, it holds that $\mathds{E}_{P_0}[|R(x, \, \ran) |] \,  \leq \,   \mathds{E}_{P_0}[\|\ran\|] \sup_{y \in \Xcal} \|x - y \|_* \, <  \, \infty$ for all $x \in \Xcal$. Given these properties, Theorem \ref{thm:strong_duality} can be applied to express the worst-case expected regret as
\begin{align}  \label{eq:a1}
         \sup_{P \in \Pcal}  \, \mathds{E}_P\left[{ R(x, \, \ran)} \right] \,  & = \,  \inf_{\lambda \geq 0} \mathds{E}_{P_0}  \left[ \sup_{z \in \Rset^n}  \big\{  R(x, \, z)  + \lambda (r - \| z - \ran\|)    \big\}    \right]. 
\end{align}
Substituting $R(x, \, z) := z^\top x - \inf_{y \in \Xcal} \{z^\top y\}$ into the right-hand side of Eq. \eqref{eq:a1} and rearranging terms, the worst-case expected regret can be rewritten as
\begin{align} \label{eq:a2}
    \inf_{\lambda \geq 0}   \bigg\{   \lambda r    +  \mathds{E}_{P_0}  \bigg[ \sup_{z \in \Rset^n} \sup_{y \in \Xcal} \big\{  z^\top (x-y)   - \lambda  \| z - \ran\|    \big\}    \bigg]    \bigg\}.
\end{align}
Next, we exchange the order of the supremum taken with respect to $z \in \Rset^n$ with the supremum taken with respect to $y \in \Xcal$ in \eqref{eq:a2}. Then, using the  fact that 
the conjugate function of  $\lambda \|\cdot \|$ is given by
\begin{align*}
    \sup_{z \in \Rset^n}  \big\{ z^\top\xi  - \lambda \|z\| \big\} = \begin{cases}
        0 & \text{if } \|\xi \|_* \leq \lambda, \\ \infty &  \text{otherwise},
    \end{cases}
\end{align*}
for all $\xi \in \Rset^n$, we can simplify the optimization problem in \eqref{eq:a2} to
\begin{align} \label{eq:a3}
  \inf_{\lambda \geq 0}   \Bigg\{   \lambda r  \  +  \ \mathds{E}_{P_0}  \Bigg[ \sup_{y \in \Xcal}    \begin{cases} \ran^\top (x - y)   &\text{if } \|x - y \|_* \leq \lambda, \\ \infty & \text{otherwise}
  \end{cases}  \Bigg\}    \Bigg]    \Bigg\}.
\end{align}
 Clearly, the value of the expectation in \eqref{eq:a3} is infinite if there exists any element $v \in \Xcal$ such that $\| x- v\|_* > \lambda$. Hence, the optimization problem \eqref{eq:a3} can be rewritten as
 \begin{align} \label{eq:a4}
    \inf_{\lambda \geq 0}   \bigg\{   \lambda r   \  +  \ \mathds{E}_{P_0}  \bigg[ \sup_{y \in \Xcal}    \big\{  \ran^\top (x-y)    \big\}    \bigg]  \ \bigg| \   \|x - v\|_* \leq \lambda  \, \text{ for all }  \, v \in \Xcal  \bigg\}.    
 \end{align}
From \eqref{eq:a4}, the optimal dual variable  is found to be $ \lambda^\star = \sup_{v \in \Xcal} \| x - v\|_* .$ Thus, the optimal value of problem \eqref{eq:a4} is given by
\begin{align} \label{eq:a5}
    r\sup_{v \in \Xcal} \| x - v\|_*   \ +  \ \mathds{E}_{P_0}  \bigg[ \sup_{y \in \Xcal}    \big\{  \ran^\top (x-y)    \big\}    \bigg] ,
\end{align}
where the second term in \eqref{eq:a5} is $\mathds{E}_{P_0}[R(x, \, \ran) ]$, proving the equivalence in \eqref{eq:worst_case_expec_reform}.  Finally,  the equivalence  of the optimal solution sets for problems  \eqref{eq:DRRO}  and \eqref{eq:DRRO_reduction} is a direct consequence of   \eqref{eq:worst_case_expec_reform}, which implies that
\begin{align}
    \inf_{x \in \Xcal} \sup_{P \in \Pcal}  \, \mathds{E}_P\left[{ R(x, \, \ran)} \right] \,  = \, \inf_{x \in \Xcal}  \Big\{ \mathds{E}_{P_0}[\ran^\top x]   \ + \ r \sup_{v \in \Xcal} \|x - v\|_*   \Big\}   \ - \ \mathds{E}_{P_0} \Big[ \inf_{y \in \Xcal} \ran^\top y  \Big],
\end{align}
thus completing the proof. \end{proof}

\begin{rem}[Regularization] \rm 
Theorem \ref{thm:DRRO} shows that the worst-case expected regret incurred by a decision $x \in \Xcal$ over all distributions in the type-1 Wasserstein ball can be expressed as the sum of two terms: the expected regret 
of the decision $x$ under the nominal distribution and a regularization term that 
penalizes the distance of the decision $x$ to the farthest point in the feasible set $\Xcal$. 
Specifically, the regularization function $\sup_{v \in \Xcal} \| x- v \|_*$ represents the radius of the smallest dual-norm ball centered at $x$ that contains the feasible set $\Xcal$.
Thus, as the radius $r$ of the ambiguity set $\Pcal$ increases, the set of optimal solutions to the distributionally robust regret minimization problem \eqref{eq:DRRO} will gravitate towards the ``center'' of the feasible set $\Xcal$,  defined by $\argmin_{x \in \Xcal} \sup_{v \in \Xcal} \| x-v\|_*$.\footnote{While this notion of a ``central point'' in the feasible set $\Xcal$ is  similar to the Chebyshev center of the set $\Xcal$, which is typically defined as the center of the smallest ball that contains the set $\Xcal$, it differs in that it requires the center of the enclosing ball to lie within the set $\Xcal$. In contrast, the Chebyshev center may not belong to $\Xcal$, even if $\Xcal$ is a convex set.} 

The regularization function $\sup_{v \in \Xcal} \| x - v \|_*$ can also be shown to coincide with the Lipschtiz modulus of the regret function $R(x, \cdot)$ on $\Rset^n$ with respect to the norm $\| \cdot\|$. ``Lipschitz regularization'' has also been shown to arise in distributionally robust optimization problems equipped with type-1 Wasserstein ambiguity sets centered at discrete (empirical) nominal distributions \cite[Theorem 6.3]{mohajerin2018data}, \cite[Theorem 10]{kuhn2019wasserstein}.
\end{rem}

\begin{rem}[Interpolation] \rm 
Interestingly, the regularization function can also be interpreted as the worst-case regret incurred by a decision $x \in \Xcal$, where the supremum is taken over all cost vectors $w$ within a norm-bounded uncertainty set $\uset = \{ w \in \Rset^n \, | \, \|w\| \leq 1\}$.  Specifically, we have that
\begin{align*}
    \sup_{\|\rel\| \leq 1} R(x, \, \rel) \,  =  \, \sup_{\|\rel\| \leq 1}\,  \sup_{y \in \Xcal}  \big\{ \rel^\top (x - y)  \big\} \,  =  \, \sup_{y \in \Xcal} \| x - y\|_*,
\end{align*}
where the last equality is obtained by exchanging the order of the supremum taken with respect to $w$ and $y$, and using the fact that $\sup_{\|w\| \leq 1} w^\top(x-y) = \|x-y\|_*$. With this interpretation,  problem \eqref{eq:DRRO_reduction} can be rewritten as
   \begin{align} \label{eq:interpolate}
\inf_{x \in \Xcal}  \Big\{ \mathds{E}_{P_0}[\ran^\top x]   \ + \ r    \sup_{\|\rel\| \leq 1} R(x, \, \rel)   \Big\}.
\end{align}
Thus, the Wasserstein distributionally robust regret minimization problem can be understood as ``interpolating'' between the  nominal expected cost minimization problem $\inf_{x \in \Xcal} \mathds{E}_{P_0}[ \ran^\top x]$  (when $r \rightarrow 0$) and the robust regret minimization problem  $\inf_{x \in \Xcal} \, \sup_{\|\rel\|  \leq 1 }  \ R(x,\, \rel)$  (when $r \rightarrow \infty$).
The radius of the ambiguity set $r$, which reflects the decision maker's confidence in the accuracy of the nominal distribution, acts as the regularization parameter that governs this tradeoff.

\end{rem}

\subsection{Complexity and Tractable Reformulation} \label{sec:tractable}
 The evaluation of the regularization function $ \sup_{v \in \Xcal} \| x - v\|_*$ in problem \eqref{eq:DRRO_reduction} amounts to solving a  norm maximization problem over the feasible set $\Xcal$.  Such problems are known to be NP-hard when $\|\cdot\|_*$ is a $p$-norm with $p \in [1, \infty)$ and $\Xcal$ is a compact, convex polytope in halfspace representation \cite[Theorem 4.1]{mangasarian1986variable}.   
 However, if the feasible set can be expressed as the convex hull of a finite set of points  given by $\Xcal = \text{conv}\{ v_1, \dots, v_m \}$,  then the norm maximization problem takes on a particularly simple form:
 $$\sup_{v \in \Xcal} \| x - v\|_* \, = \, \max_{i \in \{1, \dots, m\}} \| x - v_i\|_*.$$ 
This follows from the fact that the supremum of a convex function over a convex and compact set is attained at an extreme point of that set. With this simplification,   the distributionally robust regret minimization problem \eqref{eq:DRRO_reduction} can be reformulated as the following tractable convex  program in the variables  $x$ and  $\lambda $:
\begin{align} \label{eq:Vrep_reform}
\begin{aligned}
  &\text{minimize}    & & \mathds{E}_{P_0}[\ran^\top x]   \ + \ r\lambda \\
        & \text{subject to}  & &   x \in \Xcal,  \ \lambda  \in \Rset  \\
    & & & \|x - v_i\|_* \leq \lambda,  \ \ i = 1, \dots,m.
\end{aligned}
\end{align}

 As another example,  consider the case where  $\|\cdot\|_*$ is the $\infty$-norm. In this setting, the norm maximization  problem $ \sup_{v \in \Xcal} \| x - v\|_\infty$ can be solved by  maximizing the absolute value of each component of the vector $x-v$ separately, subject to the constraint $v \in \Xcal$. Specifically, given a (possibly nonconvex) feasible set $\Xcal \subseteq \Rset^n$, we have
 \begin{align} \label{eq:reg_reform}
    \sup_{v \in \Xcal} \| x - v\|_\infty \, = \, \max_{i \in \{1, \dots,n\}} \max \, \Big\{ \sup_{v \in \Xcal} e_i^\top (x-v), \   \sup_{v \in \Xcal} e_i^\top (v-x)  \Big\},
\end{align}
where the vectors $e_1, \dots, e_n \in \Rset^n$ denote the standard basis for $\Rset^n$.  
This suggests that when the Wasserstein distance is defined in terms of the 1-norm 
(the dual of the $\infty$-norm),  the distributionally robust regret minimization problem \eqref{eq:DRRO_reduction} can be reformulated as shown in the following corollary, which follows directly from  Theorem \ref{thm:DRRO} and the simplified form of the regularization function in Eq. \eqref{eq:reg_reform}.

\begin{cor} \label{cor:infty} \rm  Let $r \geq 0$  and assume that the set $\Xcal \subseteq \Rset^n$  is nonempty and compact. If $\|\cdot\|_*$ is the $\infty$-norm, then problem \eqref{eq:DRRO_reduction} is equivalent to the following  optimization problem in the variables $x$ and $\lambda$:
\begin{align}  \label{eq:support_reform}
\begin{aligned}
    &\text{minimize}    & & \mathds{E}_{P_0}[\ran^\top x]   \ + \ r\lambda \\
        & \text{subject to}  & &   x \in \Xcal,  \ \lambda  \in \Rset  \\
    & & &     \supp_\Xcal(e_i) - e_i^\top x   \, \leq \,   \lambda,  \ \  i = 1, \dots,n\\
    &&&   \supp_\Xcal(-e_i) + e_i^\top x   \, \leq \,   \lambda,  \ \  i = 1, \dots,n.
\end{aligned}
\end{align}
Here,  $\supp_\Xcal$ denotes the support function associated with the set $\Xcal$.
\end{cor}

While problem \eqref{eq:support_reform} is a valid  reformulation of the distributionally robust regret minimization problem \eqref{eq:DRRO_reduction} for any nonempty, compact set $\Xcal$,
the  computational tractability of problem \eqref{eq:support_reform}  depends on the underlying geometry of the feasible set $\Xcal$. For example, if  $\Xcal$ is a convex set, then problem \eqref{eq:support_reform} reduces to a convex optimization problem, where  the support function values $ \supp_\Xcal( \pm e_i) $ for  $i = 1, \dots, n$ can be computed by solving $2n$ convex optimization problems, either numerically or in closed form in certain cases.  See \cite[Section 3]{ben2015deriving}  for an overview of a variety of convex sets whose support functions  can be determined in closed form or recast as minimization problems using strong duality.

\subsection{Comparison to Distributionally Robust Optimization} \label{sec:comparison}

In this section, we compare the distributionally robust regret minimization approach with the more widely studied  distributionally robust expected cost minimimization framework---commonly referred to as distributionally robust optimization (DRO)---where the decision-maker seeks to minimize the worst-case expected cost across all distributions within a specified ambiguity set by solving: 
\begin{align} \label{eq:DR0}
        \inf_{x \in \Xcal} \, \sup_{P \in \Pcal}  \, \mathds{E}_P\big[ \ran^\top x \big].
    \end{align}
While both methods aim to enhance the robustness of decisions to distributional uncertainty, they induce distinct forms of regularization, leading to  fundamental differences in the behavior of their optimal solutions as the Wasserstein radius $r$ varies.  To illustrate this difference more precisely, note that the worst-case expected cost incurred by a decision $x \in \Xcal$ can be expressed as the sum of the expected cost under the nominal distribution and a regularization term that penalizes the dual norm of $x$, scaled by the radius of the ambiguity set $r$: 
\begin{align} \label{eq:DRO_reform}
    \sup_{P \in \Pcal}  \, \mathds{E}_P\big[ \ran^\top x \big] \,  = \, \mathds{E}_{P_0}[\ran^\top x  ] \, + \,  r  \|  x \|_* .
\end{align}
The equivalence in \eqref{eq:DRO_reform} was originally shown to hold for type-1 Wasserstein ambiguity sets centered at empirical nominal distributions \cite{mohajerin2018data, kuhn2019wasserstein}. By applying Theorem \ref{thm:strong_duality}, the reformulation  in  \eqref{eq:DRO_reform}  can also be shown to hold for any  nominal distribution $P_0 \in \Mcal(\Rset^n)$. We omit the proof here, as it closely follows the proof of Theorem \ref{thm:DRRO}.

Eq. \eqref{eq:DRO_reform} reveals that distributionally robust expected cost minimization results in a form of regularization that encourages decisions that are small in dual norm, driving the set of optimal solutions toward the origin as the ambiguity set radius $r$ increases. In contrast, the regularizer induced by the distributionally robust expected regret minimization approach  draws the set of optimal solutions toward the ``center'' of the feasible region as $r$ increases. Figure \ref{fig:simple_example}  illustrates these contrasting  behaviors with a simple two-dimensional example.

\begin{figure}[htb!]
    \centering
    \includegraphics[width=.7\columnwidth]{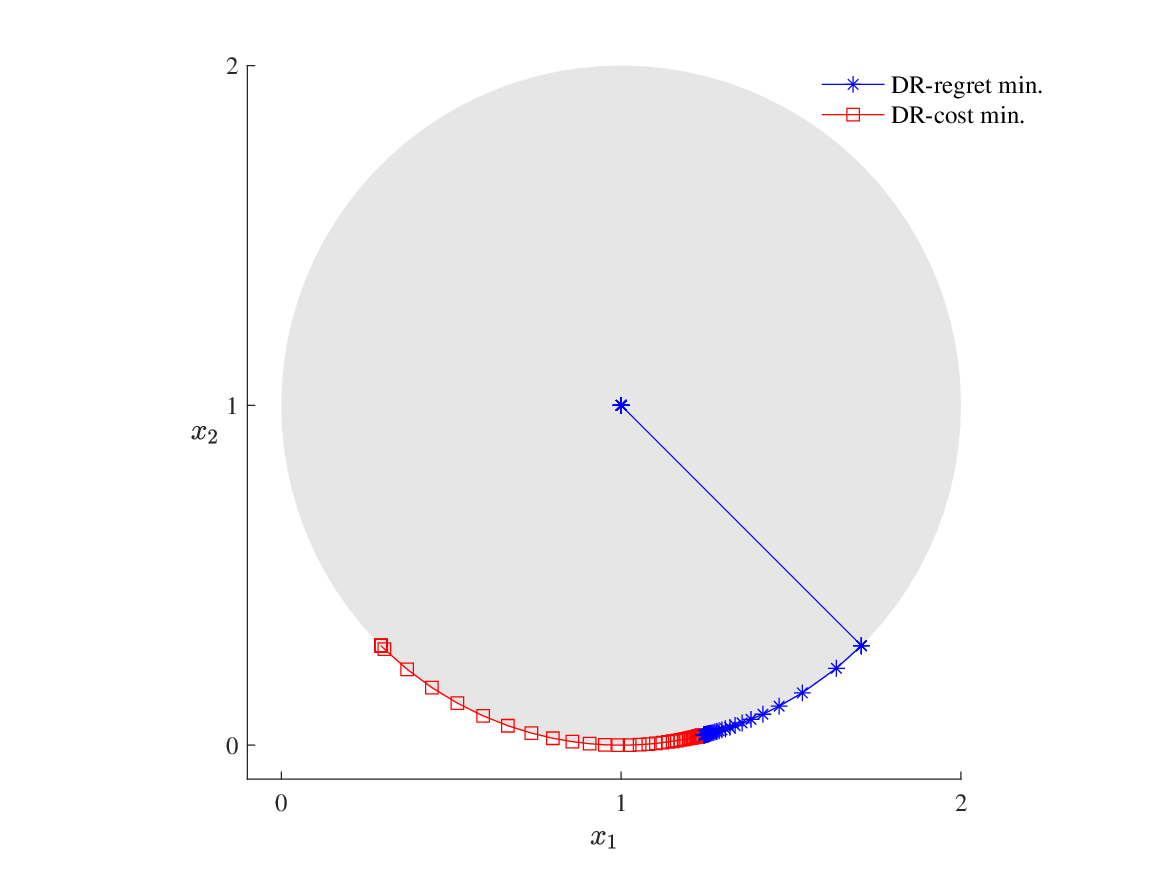}
    \caption{Illustration of optimal solutions to the  distributionally robust regret minimization problem \eqref{eq:DRRO} (blue stars) and the distributionally robust cost minimization problem \eqref{eq:DR0} (red squares)   as the ambiguity set radius $r$ ranges from $10^{-2}$ to 10 over 50 logarithmically spaced points. In this example, the nominal expected value of $\ran$ is taken to be $\mathds{E}_{P_0} [ \ran] = [-1/2, \, 2]^\top$, the dual norm $\|\cdot\|_*$ is the $\infty$-norm, and the feasible set $\Xcal = \{x \in \Rset^2 \, | \, \| x-1 \|_2 \leq 1\}$ is shown as the shaded gray region.}
    \label{fig:simple_example}
\end{figure}

\section{Worst-Case Conditional Value-at-Risk of Regret Minimization}  \label{sec:cvar}
Drawing inspiration from the decision criterion proposed in \cite{natarajan2014probabilistic},  we  now consider a risk-sensitive variant of problem \eqref{eq:DRRO}, where decisions are made to minimize the worst-case conditional value-at-risk (CVaR) of regret over the type-1 Wasserstein ball of distributions. Specifically, given a confidence level $\alpha \in [0,1)$, the conditional value-at-risk of regret incurred by a decision $x \in \Xcal$ under a distribution $P \in \Pcal$ is defined as 
\begin{align} \label{eq:cvar_regret}
    \text{CVaR}_{\alpha}^P\big(R(x, \, \ran)\big) :=  \inf_{\tau \in \Rset} \left\{\tau \, + \,  \frac{1}{1-\alpha}   \mathds{E}_P\big[   \max\{ R(x, \, \ran) - \tau, \, 0 \} \big] \right\}.
\end{align}
As a coherent risk measure \cite{artzner1999coherent}, CVaR is guaranteed to be convexity preserving.  This ensures convexity of the  (worst-case) CVaR of regret with respect to the decision variable $x$, since the regret function $R(x, \, w)$ is  linear in $x$ for each $w \in \Rset^n$.
Utilizing this class of risk measures, we formulate the worst-case CVaR of regret minimization problem as
\begin{align} \label{eq:WCVAR_min}
       \inf_{x \in \Xcal} \, \sup_{P \in \Pcal}  \,     \text{CVaR}_{\alpha}^P\big(R(x, \, \ran)\big).
\end{align}
As the confidence level  $\alpha \uparrow 1$, the CVaR of regret, $\text{CVaR}_\alpha^P\big(R(x, \, \ran)\big)$, converges to the essential supremum  of regret, defined as $ \operatorname{ess}\sup_P  \big( R(x, \, \ran) \big) :=  \inf_{\tau \in \Rset}\{ \tau \, | \, P(R(x, \, \ran) > \tau) =0 \}$.  When $\alpha = 0$, it coincides with the expected regret $\mathds{E}_P[R(x, \, \ran)]$, 
and problem \eqref{eq:WCVAR_min}  reduces to the  worst-case expected regret minimization problem \eqref{eq:DRRO}. The following result extends Theorem \ref{thm:DRRO} to encompass this more general family of risk-sensitive regret minimization problems.

\begin{thm}[Worst-Case CVaR of Regret] \label{thm:CVaR} \rm Let  $r \geq 0$ and assume that the feasible set $\Xcal \subset \Rset^n$ is nonempty and compact. Given a confidence level $\alpha \in [0, 1)$, the worst-case conditional value-at-risk of regret incurred by a decision $x \in \Xcal$ can be expressed as:
\begin{align} \label{eq:worst_case_CVaR}
    \sup_{P \in \Pcal}  \,     \text{CVaR}_{\alpha}^P\big(R(x, \, \ran)\big) \,  = \,\text{CVaR}_{\alpha}^{P_0}\big(R(x, \, \ran)\big)  \, + \,  \Big( \frac{r}{1- \alpha} \Big) \sup_{v \in \Xcal} \|  x - v \|_* .
\end{align}
\end{thm}

\begin{proof}  When $r = 0$, the result is immediate, as the ambiguity set reduces to a singleton containing only the nominal distribution $P_0$.  For $r > 0$, recall that 
\begin{align} \label{eq:z1}
    \sup_{P \in \Pcal}  \,     \text{CVaR}_{\alpha}^P\big(R(x, \, \ran)\big) \,  = \,     \sup_{P \in \Pcal}   \, \inf_{\tau \in \Rset} \left\{\tau \, + \,  \frac{1}{1-\alpha}   \mathds{E}_P\big[   \max\{ R(x, \, \ran) - \tau, \, 0 \} \big] \right\}.
\end{align}
Next, we next invoke a version of the minimax theorem to exchange the order of the infimum and supremum in \eqref{eq:z1}.  To do so, note that the objective function in \eqref{eq:z1} is convex in the variable $\tau \in \Rset$ and  linear in the probability measure $P$ over the type-1 Wasserstein ball $\Pcal$, which is a nonempty, convex, and weakly compact set \cite[Theorem 1]{yue2022linear}. 
The objective function can also be shown to be weakly continuous in $P$ on the set $\Pcal$.  Given these properties, the minimax theorem from \cite[Theorem 4.5]{borwein2016variational} can be applied to reformulate the worst-case CVaR of regret in \eqref{eq:z1} as
\begin{align} \label{eq:z2}
    \sup_{P \in \Pcal}  \,     \text{CVaR}_{\alpha}^P\big(R(x, \, \ran)\big) \,  =   \, \inf_{\tau \in \Rset} \left\{\tau \, + \,  \frac{1}{1-\alpha}   \sup_{P \in \Pcal} \,  \mathds{E}_P\big[   \max\{ R(x, \, \ran) - \tau, \, 0 \} \big] \right\}.
\end{align}
To complete the proof, one can employ Theorem \ref{thm:strong_duality} to reformulate the worst-case expectation on the right-hand side of Eq. \eqref{eq:z2} in its dual form, and follow 
a  line of reasoning similar to that used in the proof of Theorem  \ref{thm:DRRO} to show that
\begin{align*}
\sup_{P \in \Pcal} \,  \mathds{E}_P\big[   \max\{ R(x, \, \ran) - \tau, \, 0 \} \big]  \, = \, \mathds{E}_{P_0}\big[   \max\{ R(x, \, \ran) - \tau, \, 0 \} \big] \, + \,  r \sup_{v \in \Xcal} \|  x - v \|_*.
\end{align*}
Substituting the above expression into Eq. \eqref{eq:z2} and rearranging terms yields the desired result. \end{proof}
Theorem \ref{thm:CVaR} shows that for $\alpha \in [0,1)$, the worst-case CVaR of regret minimization problem \eqref{eq:WCVAR_min} can be reformulated as
\begin{align} \label{eq:CVaR_problem_reform}
    \inf_{x \in \Xcal} \,  \Big\{ \text{CVaR}_{\alpha}^{P_0}\big(R(x, \, \ran)\big)  \, + \,  \Big( \frac{r}{1- \alpha} \Big) \sup_{v \in \Xcal} \|  x - v \|_* \Big\}.
\end{align}
This reformulation  possesses  the same form of regularization identified in Theorem \ref{thm:DRRO}, but with a rescaling of the regularization parameter by a factor of $1/(1- \alpha)$, reflecting a dependence on the confidence level $\alpha$.

\begin{rem}[Computational tractability] \rm 
As expected,  the reformulation \eqref{eq:CVaR_problem_reform} reduces to the one provided in Theorem \ref{thm:DRRO} when $\alpha = 0$.  To solve problem \eqref{eq:CVaR_problem_reform} for $\alpha \in (0,1)$, one must calculate the CVaR of regret under the nominal distribution $P_0$. However, despite their convexity, computing  CVaR measures with high accuracy is known to be challenging,  even  for relatively simple cost functions and distributions.\footnote{For example, \cite[Corollary 1]{hanasusanto2016comment} shows that it is  \#P-hard to compute the expected value of the non-negative part of a linear combination of uniformly distributed random variables.}  This difficulty is avoided when $P_0$ is a discrete distribution supported on a finite set of moderate cardinality.  Together with this simplifying assumption on the support of $P_0$, the structural conditions provided in Section \ref{sec:tractable} can be utilized to reformulate problem \eqref{eq:CVaR_problem_reform} as a finite-dimensional convex program. 
\end{rem}

\bibliographystyle{plain}
\bibliography{references.bib}

\end{document}